\documentclass{article}
\usepackage{psfrag}
\usepackage[parfill]{parskip}
\usepackage{float, graphicx}
\usepackage[]{epsfig}
\usepackage{amsmath, amsthm, amssymb,}
\usepackage{epsfig}
\usepackage{verbatim}
\usepackage{multicol}
\usepackage{url}
\usepackage{latexsym}
\usepackage{mathrsfs}
\usepackage[colorlinks, bookmarks=true]{hyperref}
\usepackage{graphicx}
\usepackage{amsmath}
\usepackage{enumerate}
\usepackage[normalem]{ulem}
\usepackage{bm}
\usepackage{dsfont}
\usepackage{stmaryrd}
\usepackage{enumitem}
\usepackage{cite}
\usepackage{color}
\usepackage{xcolor}
\usepackage{hyperref}

\makeatletter

\def\@settitle{\begin{center}%
    \bfseries
 \normalfont\LARGE\@title
  \end{center}%
}
\def\@setauthors{\begin{center}%
 \normalsize\@author
  \end{center}%
}

\makeatother

\numberwithin{equation}{section}

\renewcommand{\cal}{\mathcal}

\newcommand{\cE}{{\cal E}}
\newcommand{\cG}{{\cal G}}

\newcommand{\cN}{{\cal N}}

\newcommand{\cP}{{\cal P}}

\newcommand{\cS}{{\mathcal S}}
\newcommand{\cU}{{\mathcal U}}

\newcommand{\fb}{{\mathfrak b}}
\newcommand{\fc}{{\mathfrak c}}

\newcommand{\fn}{{\mathfrak n}}

\newcommand{\bma}{{\bm{a}}}

\newcommand{\bme}{{\bm{e}}}

\newcommand{\bmn}{{\bm{n}}}

\newcommand{\bmt} {{\bm t }}
\newcommand{\bmu}{{\bm{u}}}
\newcommand{\bmv}{{\bm{v}}}
\newcommand{\bmw}{{\bm{w}}}
\newcommand{\bmx}{{\bm{x}}}

\newcommand{\bmmu}{{\bm \mu}}

\newcommand{\rd}{{\rm d}}

\newcommand{\ri}{\mathrm{i}}

\newcommand{\bF}{{\mathbb F}}
\newcommand{\bE}{\mathbb{E}}

\newcommand{\bP}{\mathbb{P}}

\newcommand{\bR}{{\mathbb R}}
\newcommand{\bS}{\mathbb S}

\newcommand{\bZ}{\mathbb{Z}}

\newcommand{\al}{\alpha}

\newcommand{\la}{\lambda}

\DeclareMathOperator{\OO}{O}
\DeclareMathOperator{\oo}{o}

\newcommand{\deq}{\mathrel{\mathop:}=} 
 
\renewcommand{\leq}{\leqslant}
\renewcommand{\geq}{\geqslant}

\newcommand{\del}{\partial}

\newcommand{\beq}{\begin{equation}}
\newcommand{\eeq}{\end{equation}}

\theoremstyle{plain} 
\newtheorem{theorem}{Theorem}[section]
\newtheorem*{theorem*}{Theorem}

\newtheorem*{lemma*}{Lemma}

\newtheorem*{corollary*}{Corollary}
\newtheorem{proposition}[theorem]{Proposition}
\newtheorem*{proposition*}{Proposition}

\newtheorem*{assumption*}{Assumption}

\newtheorem*{definition*}{Definition}

\newtheorem*{example*}{Example}
\newtheorem{remark}[theorem]{Remark}

\newtheorem*{remark*}{Remark}
\newtheorem*{remarks*}{Remarks}

\def\author#1{\par
    {\centering{\authorfont#1}\par\vspace*{0.05in}}
}

\def\titlefont{\fontsize{13}{15}\bfseries\boldmath\selectfont\centering{}}
\def\authorfont{\fontsize{13}{15}}

\let\affiliationfont\rhfont

\def\address#1{\par
    {\centering{\affiliationfont#1\par}}\par\vspace*{11pt}
}

\def\body{
\setcounter{footnote}{0}
\def\thefootnote{\alph{footnote}}
\def\@makefnmark{{$^{\rm \@thefnmark}$}}
}

\def\title#1{
    \thispagestyle{plain}
    \vspace*{-14pt}
    \vskip 79pt
    {\centering{\titlefont #1\par}}%
    \vskip 1em
}

\newcommand{\Mod}[1]{\ (\mathrm{mod}\ #1)}
\newcommand{\spn}{\mathrm{span}}

\begin{document}

\title{Invertibility of adjacency matrices for random $d$-regular directed graphs}

\vspace{1.2cm}

 \author{Jiaoyang Huang}
\address{Harvard University\\
   E-mail: jiaoyang@math.harvard.edu}

~\vspace{0.3cm}

\begin{abstract}
Let $d\geq 3$ be a fixed integer, and a prime number $p$ such that $\gcd(p,d)=1$. Let $A$ be the adjacency matrix of a random $d$-regular directed graph on $n$ vertices. We show that as a random matrix in $\bF_p$, 
\begin{align*}
\bP(\text{$A$ is singular in $\bF_p$})\leq \frac{1+\oo(1)}{p-1},
\end{align*}
as $n$ goes to infinity. As a consequence, as a random matrix in $\bR$,
\begin{align*}
\bP(\text{$A$ is singular in $\bR$})=\oo(1),
\end{align*}
as $n$ goes to infinity. This answers an open problem by Frieze \cite{MR3728474} and Vu \cite{MR2432537, MR3727622}, for random $d$-regular bipartite graphs. The proof combines a local central limit theorem and a large deviation estimate.
\end{abstract}

\section{Introduction}
The most famous combinatorial problem concerning random matrices is perhaps the ``singularity'' problem. In a standard setting, when the entries of the $n\times n$ matrix are i.i.d. Bernoulli random variables (taking values $\pm1$ with probability $1/2$), this problem was first done by Koml{\'o}s \cite{MR0221962,MR0238371}, where he showed the probability of being singular is $\OO(n^{-1/2})$.  This bound was significantly improved by
Kahn, Koml{\'o}s and Szemer{\'e}di \cite{MR1260107} to an exponential bound
\begin{align*}
\bP(\text{random Bernoulli matrix is singular})< c^n,
\end{align*}
for $c=0.999$, for $c=3/4+\oo(1)$ by Tao and Vu \cite{MR2291914}, and by Rudelson and Vershynin \cite{MR2407948} . The often conjectured optimal value of
$c$ is $1/2 + \oo(1)$, and the best known value $c=1/\sqrt{2} + \oo(1)$ is due to Bourgain, Vu and Wood  \cite{MR2557947}.
Analogous results on singularity of symmetric Bernoulli matrices
were obtained in \cite{MR3158627, MR2891529,MR2267289}.

The above question can be reformulated for the
adjacency matrices of random graphs, either directed
or undirected. Both directed and undirected graphs are abundant in real life. One of the widely studied model in the undirected random graph literature is the Erd{\H o}s-R{\'e}nyi 
graph $G(n,p)$. It was shown by Costello and Vu in \cite{MR2446482}, that the adjacency matrix of $G(n,p)$ is nonsingular with high probability whenever the edge connectivity probability $p$ is above the connectivity threshold $\ln n/n$. For directed  Erd{\H o}s-R{\'e}nyi
graph, a quantitative estimate on the smallest singular value was obtained by Basak and Rudelson in \cite{MR3620692, BR}.

Another intensively studied random graph model is the random $d$-regular graph. For the adjacency matrix of random $d$-regular graphs, its entries are no longer independent. The lack of independence poses significant difficulty for the singularity problem of random $d$-regular graphs. For undirected random $d$-regular graphs, when $d\geq n^{c}$ with any $c>0$, it follows from the bulk universality result \cite{fix2} by Landon, Sosoe and Yau, the adjacency matrix is nonsingular with high probability. For random $d$-regular directed graphs, it was first proven by Cook in \cite{MR3602844}, the adjacency matrix is nonsingular with high probability when $C\ln^2 n\leq d\leq n-C\ln^2 n$. Later in \cite{MR3545253}, it was proven by Litvak, Lytova, Tikhomirov, Tomczak-Jaegermann and Youssef that, when $C\leq d\leq n/(C\ln^2 n)$, the singularity probability is bounded by $\OO(\ln^3 d/\sqrt{d})$. Quantitative estimates on the smallest singular values were derived in \cite{Nik1, basak2018, Alex1}.

For random $d$-regular graphs, the most challenging case is when $d$ is a constant. 
In \cite{Alex2}, it was proven by Litvak, Lytova, Tikhomirov, Tomczak-Jaegermann and Youssef that the adjacency matrix of random $d$-regular directed graphs has rank at least $n-1$ with high probability. In this paper we prove that the adjacency matrix of random $d$-regular directed graphs is nonsingular with high probability. One may identify a $d$-regular directed graph with a random $d$-regular bipartite graph.
Our result answers an open problem 
first appeared in \cite[Conjecture 8.4]{MR2432537} by Vu, and later collected in \cite[Section 9, Problem 7]{MR3728474} by Frieze  and \cite[Conjecture 5.8]{MR3727622} by Vu.

One  approach to estimate the singularity probability of random matrices is to decompose the null vectors $\bS^{n-1}$ into subsets according to different structural properties, e.g., combinatorial dimension \cite{MR1260107,MR2291914}, compressible and imcompressible vectors \cite{MR2407948, Nik1,MR3602844,basak2018,MR2569075}, and statistics of jumps \cite{MR3620692, BR,MR3545253,Alex1,Alex3}.  Different from previous works, which directly study the singularity probability over $\bR$, the key new idea in this paper is to study the singularity probability of adjacency matrices over a finite field $\bF_p$. At first glance, this may seem wasteful, as we discard a great amount of information. Moreover, as a matrix over $\bF_p$, the determinant of the adjacency matrix takes value in $\bF_p$. One expects that the determinant takes value zero with probability about $1/p$. In other words, the adjacency matrix over $\bF_p$ may be singular with positive probability. However, the benefit is that, over finite field $\bF_p$ we can better understand the arithmetic structure of the null vectors, which enables us to obtain a sharp estimate of the singularity probability. We decompose the null vectors $\bF_p^n$ into two classes, the equidistributed class where each number has approximately the same density, and the non-equidistributed class. We estimate the number of adjacency matrices which have a null vector in the equidistributed class using a local central limit theorem, and the number of adjacency matrices which have a null vector in the non-equidistributed class using a large deviation estimate.  In \cite{FLMS}, Ferber, Luh, McKinley and  Samotij use a similar idea to prove resilience results for random Bernoulli matrices.

After the appearance of the current preprint, the asymptotic nonsingularity of adjacency
matrices of random $d$-regular directed and undirected graphs are proven by M{\'e}sz{\'a}ros \cite{AM}, and by 
Nguyen and Wood \cite{NW}. The work of M{\'e}sz{\'a}ros \cite{AM} studies the distribution of the sandpile group of random $d$-regular graphs, and determines the distribution of $p$-Sylow subgroup of the sandpile group. Based on \cite{AM}, Nguyen and Wood in \cite{NW}, study the distribution of the cokernels of adjacency matrices of random $d$-regular graphs, and observe that the convergence of such distributions implies asymptotic nonsingularity of the matrices.
In a forthcoming paper \cite{Huang}, we give another proof that the adjacency matrix of undirected $d$-regular graphs is nonsingular with high probability, which uses the same ideas as in this paper, however, technically more complicated. 

\noindent\textbf{Acknowledgement.} I am thankful to Elchanan Mossel and Mustazee Rahman for suggesting the problem of studying adjacency matrices of random $d$-regular graphs over finite fields. I am also grateful to Weifeng Sun for enlightening discussions, and to Nicholas Cook, Van Vu and Melanie Wood for helpful comments on the first draft of this paper. 

\subsection{Main results}

We study the configuration model of random $d$-regular directed graphs, introduced by Bollob{\'a}s in \cite{MR595929} (ideas similar to the configuration model were also presented in \cite{MR0505796,MR545196,MR527727}). By a contiguity argument, our main results also hold for other random $d$-regular directed graph models, e.g. the uniform model and the sum of $d$ random permutation matrices. For the configuration model, one generates a random $d$-regular directed graph by the following procedure:
\begin{enumerate}
\item Associate to each vertex $k \in \{1,2,\cdots,n\}$ a fiber $F_k$ of $d$ points, so that there are
$
\left| \cup_{k\in \{1,2,\cdots,n\}}F_k\right|=nd
$
points in total.
\item Select a permutation $\cP$ of the $nd$ points uniformly at random.
\item For any vertex $k \in \{1,2,\cdots,n\}$, and point $k'\in F_k$, we add a directed edge from vertex $k$ to vertex $\ell$ if the point $\cP(k')$ belongs to fiber $F_{\ell}$.
\end{enumerate}
We denote the $d$-regular directed graphs obtained from the above procedure by $\mathsf{M}_{n,d}$, which is a multiset. It is easy to see from the construction procedure that $|\mathsf{M}_{n,d}|=(nd)!$. Let $\cG\in \mathsf{M}_{n,d}$, one may identify $\cG$ with a random $d$-regular bipartite graph on $n + n$ vertices in the obvious way. We denote $A$ = $A(\cG)$ the adjacency matrix of $\cG$, i.e. $A_{k\ell}$ is the number of directed edges from vertex $k$ to vertex $\ell$.

\begin{theorem}\label{thm:smcount}
Let $d\geq 3$ be a fixed integer, and a prime number $p$ such that $\gcd(p,d)=1$. Then in $\bF_p$, 
\begin{align}\label{e:smcount}
\sum_{\bmv\in \bF_p^n\setminus{\bm0}}|\{\cG\in \mathsf{M}_{n,d}: A(\cG)\bmv=\bm0\}|= (1+\oo(1))|\mathsf{M}_{n,d}|,
\end{align}
as $n$ goes to infinity. 
\end{theorem}
If an adjacency matrix $A(\cG)$ is singular as a matrix in $\bF_p$, then we have
\begin{align*}
|\{\bmv\in \bF_p^n\setminus{\bm0}: A(\cG)\bmv=\bm0 \}|\geq p-1.
\end{align*} 
Therefore it follows from Theorem \ref{thm:smcount},
\begin{align*}
(p-1)|\{\cG\in \mathsf{M}_{n,d}: A(\cG) \text{ is singular in } \bF_p\}|
\leq \sum_{\bmv\in \bF_p^n\setminus{\bm0}}|\{\cG\in \mathsf{M}_{n,d}: A(\cG)\bmv=\bm0\}|
= (1+\oo(1))|\mathsf{M}_{n,d}|,
\end{align*}
and we obtain the next theorem.

\begin{theorem}\label{thm:spFp}
Let $d\geq 3$ be a fixed integer, and a prime number $p$ such that $\gcd(p,d)=1$. Let $A$ be the adjacency matrix of a random $d$-regular directed graph on $n$ vertices. Then as a random matrix in $\bF_p$, 
\begin{align*}
\bP(\text{$A$ is singular in $\bF_p$})\leq \frac{1+\oo(1)}{p-1},
\end{align*}
as $n$ goes to infinity. 
\end{theorem}

The entries of $A(\cG)$ are all integers. Therefore, if $A(\cG)$ is singular in $\bR$, it is also singular in any finite field $\bF_p$. The next theorem follows  by taking $p$ large in Theorem \ref{thm:spFp}.

\begin{theorem}\label{thm:spR}
Let $d\geq 3$ be a fixed integer. Let $A$ be the adjacency matrix of a random $d$-regular directed graph on $n$ vertices. Then as a random matrix in $\bR$, 
\begin{align*}
\bP(\text{$A$ is singular in $\bR$})=\oo(1),
\end{align*}
as $n$ goes to infinity. 
\end{theorem}

\begin{remark}
The probability that the adjacency matrix of a random $d$-regular directed graph is singular is at least polynomial in $1/n$. In fact, if a $d$-regular directed graph contains the subgraph in Figure \ref{f:K2d},
\begin{figure}
\center
\includegraphics[width=0.3\textwidth]{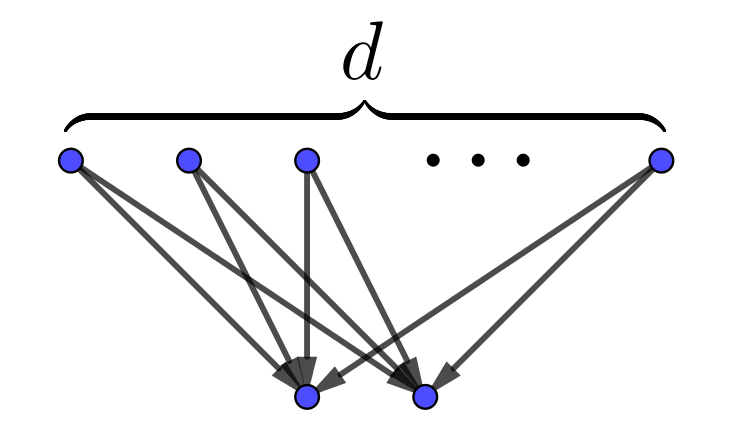}
\caption{If a $d$-regular directed graph contains the above subgraph, its adjacency matrix is singular.}
\label{f:K2d}
\end{figure}
then its adjacency matrix is singular. As a consequence, it holds that $\bP(\text{$A$ is singular in $\bR$})\geq \OO(1)/n^{d-2}$.
We can get a more quantitative estimate for $\bP(\text{$A$ is singular in $\bR$})$ in Theorem \ref{thm:spR}, by taking $p$ growing with $n$ in Theorem \ref{thm:spFp}. 
\end{remark}

\begin{remark}
Although the result in Theorem \ref{thm:spR} is stated for the configuration model, the same statement holds for other models, e.g. the uniform model and the sum of $d$ random permutation matrices,  by a contiguity argument.
\end{remark}

\section{Random Walk Interpretation}
In this section, we enumerate $|\{\cG\in \mathsf{M}_{n,d}: A(\cG)\bmv=\bm0 \text{ in } \bF_p\}|$ as the number of certain walk paths. Before stating the result, we need to introduce some notations. We define the counting function $\Phi:\cup_{k\geq 1}\bF_p^k\mapsto \bZ^p$, given by 
\begin{align*}
\Phi(a_1,a_2,\cdots, a_k)=\left(\sum_{i=1}^k \bm 1(a_i=0),\sum_{i=1}^k \bm 1(a_i=1), \cdots,\sum_{i=1}^k \bm 1(a_i=p-1)\right).
\end{align*}
We decompose the space $\bF_p^n$ as 
\begin{align*}
\bF_p^n=\bigcup_{n_0,n_1,\cdots, n_{p-1}\in \bZ_{\geq0}\atop n_0+n_1+\cdots+n_{p-1}=n}\cS(n_0,n_1,\cdots, n_{p-1}),
\end{align*}
where 
\begin{align*}
\cS(n_0,n_1,\cdots, n_{p-1})=\{\bmv=(v_1,v_2,\cdots, v_n)\in \bF_p^n: \Phi(\bmv)=(n_0,n_1,\cdots,n_{p-1})\}.
\end{align*}
The cardinality of $\cS(n_0,n_1,\cdots, n_{p-1})$ is 
\begin{align*}
\left|\cS(n_0,n_1,\cdots, n_{p-1})\right|={n\choose n_0,n_1,\cdots,n_{p-1}}.
\end{align*}
%
%
We define the multiset $\cU_{d,p}$
\begin{align}\begin{split}\label{e:defcU}
\cU_{d,p}&=\{\Phi(\bma): \bma=(a_1,a_2,\cdots,a_d)\in \bF_p^d, a_1+a_2+\cdots+a_d=0\}.
\end{split}\end{align}
For any $a_1,a_2,\cdots, a_{d-1}\in \bF_p$, there exists unique $a_d\in\bF_p$ such that $a_1+a_2+\cdots+a_d=0$. The multiset $\cU_{d,p}$ has cardinality $p^{d-1}$, i.e. $|\cU_{d,p}|=p^{d-1}$. 

\begin{proposition}\label{p:walkrep}
Let $d\geq 3$ be a fixed integer, and a prime number $p$. Fix $\bmv\in \cS(n_0,n_1,\cdots, n_{p-1})$, we have
\begin{align}\begin{split}\label{e:randomwalk}
&\phantom{{}={}}|\{\cG\in \mathsf{M}_{n,d}: A(\cG)\bmv=\bm0 \text{ in } \bF_p\}|\\
&=\left(\prod_{j=0}^{p-1}(dn_j)!\right)|\{(\bmu_1,\bmu_2\cdots, \bmu_n)\in \cU_{d,p}^n: \bmu_1+\bmu_2+\cdots+\bmu_n=(dn_0, dn_{1},\cdots, dn_{p-1})\}|\\
&=\left(\prod_{j=0}^{p-1} (dn_j)!\right)p^{n(d-1)}\bP(X_1+X_2+\cdots+X_n=(dn_0, dn_{1},\cdots, dn_{p-1})),
\end{split}\end{align}
where $X_1,X_2,\cdots, X_n$ are independent copies of $X$, which is uniform distributed over $\cU_{d,p}$.
\end{proposition}

\begin{proof}
We recall the configuration model from the introduction that each vertex $k \in \{1,2,\cdots,n\}$ is associated with a fiber $F_k$ of $d$ points.  For each permutation $\cP$ of the $nd$ points, we associate it a map $f_{\cP}: \cup_{k\in \{1,2,\cdots, n\}}F_k\mapsto \bF_p$ in the following way. For any point $k'$, if $\cP(k')=\ell'$ and $\ell'\in F_\ell$, then $f_{\cP}(k')=v_\ell$. A given map $f: \cup_{k\in \{1,2,\cdots, n\}}F_k\mapsto \bF_p$ is from a permutation if
\begin{align}\label{e:permcond}
\sum_{k\in \{1,2,\cdots, n\}}\sum_{k'\in F_k}\bm1(f(k')=j)=dn_j, \quad i=0,1,\cdots, p-1.
\end{align}
If this is the case, the number of permutation $\cP$ such that $f_{\cP}=f$ is given by
\begin{align}\label{e:numpair}
\prod_{j=0}^{p-1}(dn_j)!.
\end{align}
Let $\cG\in \mathsf{M}_{n,d}$ corresponding to a permutation $\cP$. $A(\cG)\bmv=\bm 0$ in $\bF_p$ if and only if for any $k\in \{1,2,\cdots, n\}$
\begin{align}\label{e:sum0}
\sum_{k'\in F_k}f(k')=0.
\end{align}
Especially, $\Phi(k'\in F_k: f(k')) \in \cU_{d,p}$ for any $k\in \{1,2,\cdots, n\}$. The number of maps $f: \cup_{k\in \{1,2,\cdots, n\}}F_k\mapsto \bF_p$ satisfying \eqref{e:permcond} and \eqref{e:sum0} is given by
\begin{align}\label{e:nummap}
|\{(\bmu_1,\bmu_2\cdots, \bmu_n)\in \cU_{d,p}^n: \bmu_1+\bmu_2+\cdots+\bmu_n=(dn_0, dn_{1},\cdots, dn_{p-1})\}|.
\end{align}
The claim \eqref{e:randomwalk} follows from \eqref{e:numpair} and \eqref{e:nummap}.
\end{proof}

\section{Proof of Theorem \ref{thm:smcount}}
Thanks to Proposition \ref{p:walkrep}, we can rewrite the lefthand side of \eqref{e:smcount} as
\begin{align*}
\sum_{\bmv\in \bF_p^n\setminus{\bm0}}|\{\cG\in \mathsf{M}_{n,d}: A(\cG)\bmv=\bm0\}|
&=\sum_{n_0,n_1,\cdots,n_{p-1}\in \bZ_{\geq 0}, n_0<n\atop n_0+n_1+\cdots+n_{p-1}=n}\sum_{\bmv\in \cS(n_0,n_1,\cdots, n_{p-1})} |\{\cG\in \mathsf{M}_{n,d}: A(\cG)\bmv=\bm0\}|\\
&=\sum_{n_0,n_1,\cdots,n_{p-1}\in \bZ_{\geq 0}, n_0<n\atop n_0+n_1+\cdots+n_{p-1}=n}{n\choose n_0,n_1,\cdots, n_{p-1}}\left(\prod_{j=0}^{p-1} (dn_j)!\right)\times\\
&\phantom{{}={}}\times |\{(\bmu_1,\bmu_2\cdots, \bmu_n)\in \cU_{d,p}^n: \bmu_1+\bmu_2+\cdots+\bmu_n=(dn_0, dn_{1},\cdots, dn_{p-1})\}|.
\end{align*}
Therefore Theorem \ref{thm:smcount} is equivalent to the following estiamte
\begin{align}\begin{split}\label{e:keyest}
&\phantom{{}={}}\sum_{n_0,n_1,\cdots,n_{p-1}\in \bZ_{\geq 0}, n_0<n\atop n_0+n_1+\cdots+n_{p-1}=n}{n\choose n_0,n_1,\cdots, n_{p-1}}{dn\choose dn_0,dn_1,\cdots, dn_{p-1}}^{-1}\times\\
&\times |\{(\bmu_1,\bmu_2\cdots, \bmu_n)\in \cU_{d,p}^n: \bmu_1+\bmu_2+\cdots+\bmu_n=(dn_0, dn_{1},\cdots, dn_{p-1})\}|=1+\oo(1).
\end{split}\end{align}

To prove \eqref{e:keyest}, we fix a large number $\fb>0$, and decompose those $p$-tuples $(n_0,n_1,\cdots,n_{p-1})$ into two classes:
\begin{enumerate}
\item (Equidistributed) $\cE$ is the set of $p$-tuples $(n_0,n_1,\cdots,n_{p-1})\in \bZ_{\geq 0}^n$, such that $\sum_{j=0}^{p-1}(n_j/n-1/p)^2\leq \fb\ln n/n$.
\item (Non-equidistributed) $\cN$ is the set of $p$-tuples $(n_0,n_1,\cdots,n_{p-1})\in \bZ_{\geq 0}^n$, which are not $(n,0,0,\cdots, 0)$ or equidistributed.
\end{enumerate}
In Section \ref{subs:lclt}, we estimate the sum of terms in \eqref{e:keyest} corresponding to equidistributed $p$-tuples using a local central limit theorem. In Section \ref{subs:ldp}, we show that the sum of terms in \eqref{e:keyest} corresponding to non-equidistributed $p$-tuples is small, via a large deviation estimate. Theorem \ref{thm:smcount} follows from combining Proposition \ref{p:lclt} and Proposition \ref{p:ldp}.

\subsection{Local central limit theorem estimate}
\label{subs:lclt}

In this section, we estimate the sum of terms in \eqref{e:keyest} corresponding to equidistributed $p$-tuples $(n_0,n_1,\cdots,n_{p-1})$, using a local central limit theorem.
\begin{proposition}\label{p:lclt}
Let $d\geq 3$ be a fixed integer, and a prime number $p$ such that $\gcd(p,d)=1$. Then
\begin{align}\label{e:equ}
\sum_{(n_0,n_1,\cdots,n_{p-1}) \in \cE}\sum_{\bmv\in \cS(n_0,n_1,\cdots, n_{p-1})} |\{\cG\in \mathsf{M}_{n,d}: A(\cG)\bmv=\bm0\}|=\left(1+\OO\left(\frac{(\ln n)^{3/2}}{\sqrt{n}}\right)\right)|\mathsf{M}_{n,d}|.
\end{align}
\end{proposition}

\begin{proof}
Thanks to Proposition \ref{p:walkrep}, we have
\begin{align}\begin{split}\label{e:explclt}
&\phantom{{}={}}\frac{1}{|\mathsf{M}_{n,d}|}\sum_{(n_0,n_1,\cdots,n_{p-1}) \in \cE}\sum_{\bmv\in \cS(n_0,n_1,\cdots, n_{p-1})} |\{\cG\in \mathsf{M}_{n,d}: A(\cG)\bmv=\bm0\}|\\
&=\sum_{(n_0,n_1,\cdots,n_{p-1})\in \cE}{n\choose n_0,n_1,\cdots, n_{p-1}}{dn\choose dn_0,dn_1,\cdots, dn_{p-1}}^{-1}p^{n(d-1)}\\
&\phantom{{}={}}\times\bP(X_1+X_2+\cdots+X_n=(dn_0, dn_{1},\cdots, dn_{p-1})),
\end{split}\end{align}
where $X_1,X_2,\cdots, X_n$ are independent copies of $X$, which is uniform distributed over $\cU_{d,p}$ as defined in \eqref{e:defcU}.
For an equidistributed $p$-tuple $(n_0,n_1,\cdots,n_{p-1})$, we denote $\fn_j=n_j/n$ for $j=0,1,\cdots, p-1$. Then by our definition, we have 
$\sum_{j=0}^{p-1}(\fn_j-1/p)^2\leq \fb\ln n/n$. We estimate the first factor on the righthand side of \eqref{e:explclt} using Stirling's formula,
\begin{align}\begin{split}\label{e:firstfactor}
&\phantom{{}={}}{n\choose n_0,n_1,\cdots, n_{p-1}}{dn\choose dn_0,dn_1,\cdots, dn_{p-1}}^{-1}p^{(d-1)n}\\
&=\left(1+\OO\left(\frac{1}{n}\right)\right)d^{\frac{p-1}{2}}
\exp\left\{(d-1)n\left(\sum_{j=0}^{p-1}\fn_j\ln \fn_j +\ln p\right)\right\}\\
&=\left(1+\OO\left(\frac{(\ln n)^{3/2}}{\sqrt{n}}\right)\right)d^{\frac{p-1}{2}}
\exp\left\{\frac{(d-1)pn}{2}\sum_{j=0}^{p-1}(\fn_j-1/p)^2\right\}.
\end{split}\end{align}

In the following, we estimate $\bP(S_n=(dn_0, dn_{1},\cdots, dn_{p-1}))$, where $S_n=X_1+X_2+\cdots+X_n$.
We recall that $X_1, X_2,\cdots, X_n$ are independent copies of $X$, which is uniformly distributed over the multiset $\cU_{d,p}$. The mean of $X$ is given by
\begin{align}\label{e:mean}
\bE[X(j)]
=\frac{1}{p^{d-1}}\sum_{(a_1,a_2,\cdots,a_d)\in \bF_p^d\atop a_1+a_2+\cdots+a_d=0}\sum_{k=1}^d\bm1(a_k=j)=\frac{d}{p}.
\end{align}
The covariance of $X$ is given by
\begin{align}\begin{split}\label{e:cov}
&\phantom{{}={}}\bE[(X(j)-d/p)(X(j')-d/p)]
=\frac{1}{p^{d-1}}\sum_{(a_1,a_2,\cdots,a_d)\in \bF_p^d\atop a_1+a_2+\cdots+a_d=0}\sum_{1\leq k,k'\leq d}\bm1(a_k=j)\bm1(a_{k'}=j')-\frac{d^2}{p^2}\\
&=\frac{1}{p^{d-1}}\sum_{(a_1,a_2,\cdots,a_d)\in \bF_p^d\atop a_1+a_2+\cdots+a_d=0}\left(\delta_{jj'}\sum_{1\leq k\leq d}\bm1(a_k=j)+\sum_{1\leq k\neq k'\leq d}\bm1(a_k=j)\bm1(a_{k'}=j')\right)-\frac{d^2}{p^2}=\frac{d}{p}\delta_{jj'}-\frac{d}{p^2}.
\end{split}\end{align}
We summarize \eqref{e:mean} and \eqref{e:cov} as
\begin{align}\label{e:meanandcov}
\bmmu\deq \bE[X]=(d/p,d/p,\cdots,d/p),\quad \Sigma\deq\bE[(X-\bmmu)(X-\bmmu)^t]=\frac{d}{p}I_p-\frac{d}{p^2}\bm1\bm1^t.
\end{align}
We denote the characteristic function of $X$ as
\begin{align*}
\phi_{X}(\bmt)=\bE[\exp\{\ri\langle\bmt, X\rangle\}],\quad \phi_{X-\bmmu}(\bmt)=\bE[\exp\{\ri\langle\bmt, X-\bmmu\rangle\}]=e^{-\ri \langle \bmt,\bmmu\rangle}\phi_{X}(\bmt).
\end{align*}

For a $p$-tuple $(n_0,n_1,\cdots,n_{p-1})$, if $\sum_{j=0}^{p-1}jn_j\not\equiv 0 \Mod p$, then $\bP(S_n=(dn_0,dn_1,\cdots, dn_{p-1}))=0$. We only need to consider $p$-tuples $(n_0,n_1,\cdots,n_{p-1})$ such that $\sum_{j=0}^{p-1}jn_j\equiv 0 \Mod p$. We denote $\bmn=(n_0,n_1,\cdots, n_{p-1})$. By inverse Fourier formula
\begin{align*}\begin{split}
\bP(S_n=d\bmn)
&=\frac{1}{(2\pi)^p}\int_{2\pi \bR^p/\bZ^p}\phi_X^n(\bmt)e^{-\ri \langle \bmt, d\bmn\rangle}\rd \bmt\\
&=\frac{1}{(2\pi)^p}\int_{2\pi \bR^p/\bZ^p}\phi_{X-\bmmu}^n(\bmt)e^{-\ri \langle \bmt, d\bmn-n\bmmu\rangle}\rd \bmt.
\end{split}\end{align*}
The lattice spanned by vectors in $\cU_{d,p}$ is the dual lattice of $\spn \{(0,1/p, 2/p,\cdots, p-1/p), \bme_1,\bme_2,\cdots, \bme_{p}\}$
in $\{(x_1, x_2,\cdots, x_p)\in \bR^p: x_1+x_2+\cdots+x_p=d\}$, where $\bme_1,\bme_2,\cdots, \bme_{p}$ is the standard base of $\bR^p$. Therefore $|\phi_{X-\bmmu}^n(\bmt)|=1$ if any only if 
\begin{align}\label{e:line}
\bmt\in2\pi(0,1/p,2/p,\cdots, p-1/p)\bZ+(1,1,\cdots,1)\bR.
\end{align}
For $\bmt$ which are away from those lines in \eqref{e:line}, the characteristic function $\phi^n_{X-\bmmu}(\bmt)$ is exponentially small. 
We define domains 
\begin{align*}
B_j(\delta)=2\pi j(0,1/p,2/p,\cdots, (p-1)/p)+Q (\{\bmx\in \bR^{p-1}:\|\bmx\|_2^2\leq \delta\}\times [0,2\sqrt{p}\pi]), \quad j=0,1,2,\cdots, p-1,
\end{align*}
where $Q$ is a $p\times p$ orthogonal matrix  $Q=[O, \bm1/\sqrt{p}]$.
From the discussion above, we get
\begin{align}\begin{split}\label{e:integ1}
\bP(S_n=d\bmn)
&=\frac{1}{(2\pi)^p}\sum_{j=0}^{p-1}\int_{2\pi B_j(\delta)}\phi_{X-\bmmu}^n(\bmt)e^{-\ri \langle \bmt, d\bmn-n\bmmu\rangle}\rd \bmt + e^{-c(\delta)n}\\
&=\frac{p}{(2\pi)^p}\int_{2\pi B_0(\delta)}\phi_{X-\bmmu}^n(\bmt)e^{-\ri \langle \bmt, d\bmn-n\bmmu\rangle}\rd \bmt + e^{-c(\delta)n},
\end{split}\end{align}
where we used the fact that the integrand is translation invariant by vectors $2\pi(0,1/p, 2/p,\cdots, p-1/p)\bZ$.
For any $\bmt\in  B_0(\delta)$, by definition there exists $\bmx\in \bR^{p-1}$ with $\|\bmx\|^2_2\leq \delta$ and $y\in [0,2\sqrt{p}\pi]$, such that $\bmt=Q(\bmx,y)=O\bmx+(y/\sqrt{p})\bm1$. By a change of variable, we can rewrite \eqref{e:integ1} as
\begin{align}\begin{split}\label{e:integ2}
&\phantom{{}={}}\frac{p}{(2\pi)^p}\int_{2\pi B_0(\delta)}\phi_{X-\bmmu}^n(\bmt)e^{-\ri \langle \bmt, d\bmn-n\bmmu\rangle}\rd \bmt\\
&=\frac{p}{(2\pi)^{p}}\int_{\{\bmx\in \bR^{p-1}:\|\bmx\|^2_2\leq \delta\}\times [0,2\sqrt{p}\pi]}\phi_{X-\bmmu}^n(Q(\bmx,y))e^{-\ri \langle Q(\bmx,y), d\bmn-n\bmmu\rangle}\rd \bmx\rd y\\
&=\frac{p^{3/2}}{(2\pi)^{p-1}}\int_{\{\bmx\in \bR^{p-1}:\|\bmx\|^2_2\leq \delta\}}\phi_{X-\bmmu}^n(O\bmx)e^{-\ri \langle O\bmx, d\bmn-n\bmmu\rangle}\rd \bmx,
\end{split}\end{align}
where we used that $\langle \bm1, X-\bmmu\rangle=0$ and $\langle\bm1, d\bmn-n\bmmu\rangle=0$. By Taylor expansion, the characteristic function is
\begin{align}\begin{split}\label{e:cf}
\phi_{X-\bmmu}(O\bmx)
&=\bE\left[1+\ri\langle O\bmx, X-\bmmu\rangle-\frac{1}{2}\langle O\bmx, X-\bmmu\rangle^2+\OO(\|\bmx\|_2^3)\right]\\
&=1-\frac{1}{2}\bmx^tO^t\Sigma O\bmx+\OO(\|\bmx\|_2^3)
=1-\frac{d}{2p}\|\bmx\|_2^2+\OO(\|\bmx\|_2^3),
\end{split}\end{align}
where we used $\Sigma=dI_p/p-d\bm1\bm1^t/p^2$ from \eqref{e:meanandcov}, and $O^t\Sigma O=d I_{p-1}/p$. Fix a large constant $\fc$, which will be chosen later. 
For $\fc\ln n/n\leq \|\bmx\|_2^2\leq \delta$, we have
\begin{align}\label{e:integ3}
|\phi_{X-\bmmu}(O\bmx)|^n\leq \exp\left\{-\left(\frac{\fc d}{2p}+\oo(1)\right)\ln n\right\},
\end{align}
which turns out to be negligible provided $\fc$ is large enough. In the following we will restrict the integral \eqref{e:integ2} on the domain 
$\{\bmx\in \bR^{p-1}:\|\bmx\|^2_2\leq \fc\ln n/n\}$. From \eqref{e:cf}, on the domain $\{\bmx\in \bR^{p-1}:\|\bmx\|^2_2\leq \fc\ln n/n\}$, we have 
\begin{align*}
\phi_{X-\bmmu}^n(O\bmx)=\left(1+\OO\left(\frac{(\ln n)^{3/2}}{n^{1/2}}\right)\right)e^{-\frac{dn}{2p}\|\bmx\|^2_2},
\end{align*}
and
\begin{align}\begin{split}\label{e:integ4}
&\phantom{{}={}}\frac{p^{3/2}}{(2\pi)^{p-1}}\int_{\{\bmx\in \bR^{p-1}:\|\bmx\|^2_2\leq \fc\ln n/n\}}\phi_{X-\bmmu}^n(O\bmx)e^{-\ri \langle O\bmx, d\bmn-n\bmmu\rangle}\rd \bmx\\
&=\left(1+\OO\left(\frac{(\ln n)^{3/2}}{n^{1/2}}\right)\right)\frac{p^{3/2}}{(2\pi)^{p-1}}\int_{\{\bmx\in \bR^{p-1}:\|\bmx\|^2_2\leq \fc\ln n/n\}}e^{-\frac{dn}{2p}\|\bmx\|^2_2}e^{-\ri \langle \bmx, O^t(d\bmn-n\bmmu)\rangle}\rd \bmx\\
&=\left(1+\OO\left(\frac{(\ln n)^{3/2}}{n^{1/2}}\right)\right)\frac{p^{3/2}}{(2\pi)^{p-1}}\int_{ \bR^{p-1}}e^{-\frac{dn}{2p}\|\bmx\|^2_2}e^{-\ri \langle \bmx, O^t(d\bmn-n\bmmu)\rangle}\rd \bmx+e^{-\left(\frac{\fc d}{2p}+\oo(1)\right)\ln n}\\
&=\left(1+\OO\left(\frac{(\ln n)^{3/2}}{n^{1/2}}\right)\right)\frac{p^{3/2}}{(2\pi)^{p-1}}\int_{ \bR^{p-1}}e^{-\frac{dn}{2p}\|\bmx\|^2_2}e^{-\ri \langle \bmx, O^t(d\bmn-n\bmmu)\rangle}\rd \bmx+e^{-\left(\frac{\fc d}{2p}+\oo(1)\right)\ln n}\\
&=\left(1+\OO\left(\frac{(\ln n)^{3/2}}{n^{1/2}}\right)\right)p^{3/2}
\left(\frac{p}{2\pi d n}\right)^{\frac{p-1}{2}}e^{-\frac{dpn}{2}\left\|O^t \left(\frac{\bmn}{n}-\frac{\bmmu}{d}\right)\right\|_2^2}
+e^{-\left(\frac{\fc d}{2p}+\oo(1)\right)\ln n}.
\end{split}\end{align}
The exponents in \eqref{e:firstfactor} and \eqref{e:integ4} cancel
\begin{align*}
-\frac{dpn}{2}\left\|O^t \left(\frac{\bmn}{n}-\frac{\bmmu}{d}\right)\right\|_2^2+\frac{(d-1)pn}{2}\sum_{j=0}^{p-1}(\fn_j-1/p)^2
=-\frac{pn}{2}\left\|O^t \left(\frac{\bmn}{n}-\frac{\bmmu}{d}\right)\right\|_2^2,
\end{align*}
where we used that $O^t$ is an isometry from $\{(x_1,x_2,\cdots, x_{p})\in \bR^p: x_1+x_2+\cdots+x_p=0\}$ to $\bR^{p-1}$.
Therefore, by combining the estimates \eqref{e:firstfactor}, \eqref{e:integ1}, \eqref{e:integ2}, \eqref{e:integ3} and \eqref{e:integ4}, we conclude that for any $p$-tuple $(n_0,n_1,\cdots,n_{p-1})\in \cE$, with $\sum_{j=0}^{p-1}jn_j\equiv 0 \Mod p$,
\begin{align}\begin{split}\label{e:integ4.5}
&\phantom{{}={}}\frac{1}{|\mathsf{M}_{d,p}|}\sum_{\bmv\in \cS(n_0,n_1,\cdots, n_{p-1})} |\{\cG\in \mathsf{M}_{n,d}: A(\cG)\bmv=\bm0\}|\\
&=
\left(1+\OO\left(\frac{(\ln n)^{3/2}}{n^{1/2}}\right)\right)
p^{3/2}
\left(\frac{p}{2\pi n}\right)^{\frac{p-1}{2}}e^{-\frac{pn}{2}\left\|O^t \left(\frac{\bmn}{n}-\frac{\bmmu}{d}\right)\right\|_2^2}+e^{-\left(\frac{\fc d}{2p}-\frac{(d-1)\fb p}{2}+\oo(1)\right)\ln n}.
\end{split}\end{align}
For the second term on the righthand side of \eqref{e:integ4.5}, since the total number of $p$-tuples $(n_0,n_1,\cdots, n_{p-1})\in \cE$ is bounded by $e^{(1+\oo(1))p\ln n/2}$,
\begin{align}\label{e:secondterm}
\sum_{(n_0,n_1,\cdots,n_{p-1}) \in \cE}e^{-\left(\frac{\fc d}{2p}-\frac{(d-1)\fb p}{2}+\oo(1)\right)\ln n}
=e^{-\left(\frac{\fc d}{2p}-\frac{(\fb d-\fb+1) p}{2}+\oo(1)\right)\ln n},
\end{align}
which is negligible provided $\fc$ is large enough.

For the term on the righthand side of \eqref{e:integ4.5} corresponding to the $p$-tuple $(n_0,n_1,\cdots,n_{p-1})\in \cE$, with $\sum_{j=0}^{p-1}jn_j\equiv 0 \Mod p$, we can replace it by an average.
\begin{align}\begin{split}\label{e:ave}
pe^{-\frac{pn}{2}\left\|O^t \left(\frac{\bmn}{n}-\frac{\bmmu}{d}\right)\right\|_2^2}
&=e^{-\frac{pn}{2}\left\|O^t \left(\frac{\bmn}{n}-\frac{\bmmu}{d}\right)\right\|_2^2}
+\left(1+\OO\left(\frac{(\ln n)^{1/2}}{n^{1/2}}\right)\right)
\sum_{j=1}^{p-1}e^{-\frac{pn}{2}\left\|O^t \left(\frac{\bmn+\bme_j-\bme_0}{n}-\frac{\bmmu}{d}\right)\right\|_2^2}.
\end{split}\end{align}
Therefore, we can replace the sum over $p$-tuples $(n_0,n_1,\cdots,n_{p-1})\in \cE$, with $\sum_{j=0}^{p-1}jn_j\equiv 0 \Mod p$ to the sum over all $p$-tuples $(n_0,n_1,\cdots,n_{p-1})\in \cE$ with a factor $1/p$. 
\begin{align}\begin{split}\label{e:integ5}
&\phantom{{}={}}\sum_{(n_0,n_1,\cdots,n_{p-1})\in \cE\atop \sum_{j=0}^{p-1}jn_j\equiv 0\Mod p}
\left(1+\OO\left(\frac{(\ln n)^{3/2}}{n^{1/2}}\right)\right)
p^{3/2}
\left(\frac{p}{2\pi n}\right)^{\frac{p-1}{2}}e^{-\frac{pn}{2}\left\|O^t \left(\frac{\bmn}{n}-\frac{\bmmu}{d}\right)\right\|_2^2}\\
&=\sum_{(n_0,n_1,\cdots,n_{p-1})\in \cE}
\left(1+\OO\left(\frac{(\ln n)^{3/2}}{n^{1/2}}\right)\right)
p^{1/2}
\left(\frac{p}{2\pi n}\right)^{\frac{p-1}{2}}e^{-\frac{pn}{2}\left\|O^t \left(\frac{\bmn}{n}-\frac{\bmmu}{d}\right)\right\|_2^2}.
\end{split}\end{align}
In the following we estimate the sum in \eqref{e:integ5}. The set of points $O^t(\bmn/n-\bm\mu/d)$ for $\bmn=(n_0,n_1,\cdots, n_{p-1})\in \cE$ is a subset of a lattice in $\bR^{p-1}$. A set of base for this lattice is given by
\begin{align*}
O^t(e_{j}-e_{0})/n, \quad 0\leq j\leq p-1.
\end{align*}
The volume of the fundamental domain is $p^{1/2}/n^{p-1}$. By viewing \eqref{e:integ5} as a Riemann sum, we can rewrite it as an integral on $\bR^{p-1}$.
\begin{align}\begin{split}\label{e:integ6}
&\phantom{{}={}}\sum_{(n_0,n_1,\cdots,n_{p-1})\in \cE}
\left(1+\OO\left(\frac{(\ln n)^{3/2}}{n^{1/2}}\right)\right)
p^{1/2}
\left(\frac{p}{2\pi n}\right)^{\frac{p-1}{2}}e^{-\frac{pn}{2}\left\|O^t \left(\frac{\bmn}{n}-\frac{\bmmu}{d}\right)\right\|_2^2}\\
&=\left(1+\OO\left(\frac{(\ln n)^{3/2}}{n^{1/2}}\right)\right)
\left(\frac{pn}{2\pi}\right)^{\frac{p-1}{2}}\int_{\{\bmx\in\bR^{p-1}:\|\bmx\|_2^2\leq \fb\ln n/n\}} e^{-\frac{pn}{2}\|\bmx\|^2}\rd \bmx\\
&=\left(1+\OO\left(\frac{(\ln n)^{3/2}}{n^{1/2}}\right)\right),
\end{split}\end{align}
provided $\fb$ is large enough. The claim \eqref{e:equ} follows from combining \eqref{e:integ4.5}, \eqref{e:secondterm} and \eqref{e:integ6}.
%
%
This finishes the proof of Proposition \ref{p:lclt}.
\end{proof}

\subsection{Large deviation estimate}
\label{subs:ldp}

In this section, we show that the sum of terms in \eqref{e:keyest} corresponding to non-equidistributed $p$-tuples $(n_0,n_1,\cdots,n_{p-1})$ is small. 
\begin{proposition}\label{p:ldp}
Let $d\geq 3$ be a fixed integer, and a prime number $p$ such that $\gcd(p,d)=1$. Then
\begin{align}\label{e:noneq}
\frac{1}{|\mathsf{M}_{n,d}|}\sum_{(n_0,n_1,\cdots,n_{p-1}) \in \cN}\sum_{\bmv\in \cS(n_0,n_1,\cdots, n_{p-1})} |\{\cG\in \mathsf{M}_{n,d}: A(\cG)\bmv=\bm0\}|\leq \frac{\OO(1)}{n^{(d-2)}}.
\end{align}
\end{proposition}
Thanks to Proposition \ref{p:walkrep}, we have
\begin{align}\begin{split}\label{e:expldp}
&\phantom{{}={}}\frac{1}{|\mathsf{M}_{n,d}|}\sum_{(n_0,n_1,\cdots,n_{p-1}) \in \cN}\sum_{\bmv\in \cS(n_0,n_1,\cdots, n_{p-1})} |\{\cG\in \mathsf{M}_{n,d}: A(\cG)\bmv=\bm0\}|\\
&=\sum_{(n_0,n_1,\cdots,n_{p-1})\in \cN}{n\choose n_0,n_1,\cdots, n_{p-1}}{dn\choose dn_0,dn_1,\cdots, dn_{p-1}}^{-1}\\
&\phantom{{}={}}\times |\{(\bmu_1,\bmu_2\cdots, \bmu_n)\in \cU_{d,p}^n: \bmu_1+\bmu_2+\cdots+\bmu_n=(dn_0, dn_{1},\cdots, dn_{p-1})\}|
\end{split}\end{align}
where the multiset $\cU_{d,p}$ is defined in \eqref{e:defcU}. We enumerate the elements of $\cU_{d,p}$ as
\begin{align*}
\cU_{d,p}=\{\bmw_1, \bmw_2,\cdots, \bmw_{p^{d-1}}\}, \quad \bmw_1=(d,0,0,\cdots,0).
\end{align*}
for $2\leq j\leq p^{d-1}$ we have that $\bmw_j(0)\leq d-2$ and $\bmw_j(1)+\bmw_j(2)+\cdots \bmw_j(p-1)\geq 2$. For any non-equidistributed $p$-tuple $(n_0,n_1,\cdots,n_{p-1})$, we denote $\fn_j=n_j/n$ for $j=0,1,\cdots, p-1$. We estimate the first factor on the righthand side of \eqref{e:expldp} using Stirling's formula,
\begin{align}\begin{split}\label{e:firstfactor2}
{n\choose n_0,n_1,\cdots, n_{p-1}}{dn\choose dn_0,dn_1,\cdots, dn_{p-1}}^{-1}
\leq e^{\OO(\ln n)}\exp\{(d-1)n\sum_{j=0}^{p-1} \fn_j\ln \fn_j\}.
\end{split}\end{align}
For the number of walk paths in \eqref{e:expldp}, we have the following bound
\begin{align}\begin{split}\label{e:path}
&\phantom{{}={}}|\{(\bmu_1,\bmu_2\cdots, \bmu_n)\in \cU_{d,p}^n: \bmu_1+\bmu_2+\cdots+\bmu_n=(dn_0, dn_{1},\cdots, dn_{p-1})\}|\\
&=\sum_{a_1+a_2+\cdots+a_{p^{d-1}}=n,\atop a_1\bmw_1+a_2\bmw_2+\cdots a_{p^{d-1}}\bmw_{p^{d-1}}=d(n_0,n_1,\cdots,n_{p-1})} {n\choose a_1,a_2,\cdots,a_{p^{d-1}}}\\
&\leq e^{\OO(\ln n)}\sup_{a_1\bmw_1+a_2\bmw_2+\cdots a_{p^{d-1}}\bmw_{p^{d-1}}=d(n_0,n_1,\cdots,n_{p-1})} {n\choose a_1,a_2,\cdots,a_{p^{d-1}}}\\
&\leq e^{\OO(\ln n)}\sup_{\al_1\bmw_1+\al_2\bmw_2+\cdots \al_{p^{d-1}}\bmw_{p^{d-1}}=d(\fn_0,\fn_1,\cdots,\fn_{p-1})} \exp\left\{-n\sum_{j=1}^{p^{d-1}}\alpha_j\ln \alpha_j\right\}.
\end{split}\end{align}
From estimates \eqref{e:firstfactor} and \eqref{e:path}, the term in \eqref{e:expldp} corresponding to the $p$-tuple $(n_0,n_1,\cdots,n_{p-1})$ is exponentially small, if the rate function 
\begin{align}\label{e:rate}
\sup_{\al_1\bmw_1+\al_2\bmw_2+\cdots \al_{p^{d-1}}\bmw_{p^{d-1}}=d(\fn_0,\fn_1,\cdots,\fn_{p-1})} -\sum_{j=1}^{p^{d-1}}\alpha_j\ln \alpha_j+(d-1)\sum_{j=0}^{p-1}\fn_j\ln \fn_j,
\end{align}
is negative. The following proposition states that the rate function \eqref{e:rate} is negative except for two points.
\begin{proposition}\label{p:ldpbound}
Let $d\geq 3$ be a fixed integer, and a prime number $p$ such that $\gcd(p,d)=1$. For any $\alpha_1,\alpha_2,\cdots,\alpha_{p^{d-1}}\geq 0$ such that 
\begin{align}\label{e:constraint}
\al_1\bmw_1+\al_2\bmw_2+\cdots \al_{p^{d-1}}\bmw_{p^{d-1}}=d(\fn_0,\fn_1,\cdots,\fn_{p-1}),
\end{align}
we have
\begin{align}\label{e:neg}
-\sum_{j=1}^{p^{d-1}}\alpha_j\ln \alpha_j+(d-1)\sum_{j=0}^{p-1}\fn_j\ln \fn_j\leq 0.
\end{align}
Equality holds in the following two points:
\begin{enumerate}\label{e:tomax}
\item $\alpha_1=\alpha_2=\cdots=\alpha_{p^{d-1}}=1/p^{(d-1)}$, $\fn_0=\fn_1=\cdots=\fn_{p-1}=1/p$.
\item $\alpha_1=1$, $\alpha_2=\cdots=\alpha_{p^{d-1}}=0$, $\fn_0=1$, $\fn_1=\cdots=\fn_{p-1}=0$.
\end{enumerate}
\end{proposition}
\begin{proof}
We view the constraint \eqref{e:constraint} as a defining relation for $(\fn_0,\fn_1,\cdots, \fn_{p-1})$, as a function of $\alpha_1,\alpha_2,\cdots,\alpha_{p^{d-1}}$ with the constraint $\alpha_1+\alpha_2+\cdots+\alpha_{p^{d-1}}=1$. Then we can find the extreme points of the lefthand side of \eqref{e:neg} by the method Lagrange multipliers. We define the function 
\begin{align*}
f(\alpha_1,\alpha_2,\cdots, \alpha_{p^{d-1}},\la)=-\sum_{j=1}^{p^{d-1}}\alpha_j\ln \alpha_j+(d-1)\sum_{j=0}^{p-1}\fn_j\ln \fn_j+\la\left(\sum_{j=1}^{p^{d-1}}\alpha_j-1\right).
\end{align*}
Let $\del_{\alpha_j} f=0$, we get that extreme points satisfy
\begin{align}\label{e:character}
-\ln \alpha_j -1 +\frac{d-1}{d}\sum_{k=0}^{p-1}\left(\bmw_j(k)\ln \fn_k+\bmw_j(k)\right)+\la=0.
\end{align}
We notice that $\sum_{k=0}^{p-1}\bmw_j(k)=d$ for $1\leq j\leq p^{d-1}$. By multiplying both sides of \eqref{e:character} by $\alpha_j$, and summing over $j=1,2,\cdots p^{d-1}$, we get
\begin{align}\label{e:goal}
-\sum_{j=1}^{p^{d-1}}\alpha_j\ln \alpha_j+(d-1)\sum_{j=0}^{p-1}\fn_j\ln \fn_j=-(d-2+\la).
\end{align}
It follows by rearranging \eqref{e:character}, we get
\begin{align*}
\alpha_j=e^{d-2+\la}\prod_{k=0}^{p-1}\fn_k^{\frac{d-1}{d}\bmw_j(k)}.
\end{align*}
Therefore, $\la$ satisfies
\begin{align*}
1=\sum_{j=1}^{p^{d-1}}\alpha_j
=e^{d-2+\la}\sum_{j=1}^{p^{d-1}}\prod_{k=0}^{p-1}\fn_k^{\frac{d-1}{d}\bmw_j(k)}.
\end{align*}
By the defining relation of the multiset $\cU_{d,p}$ as in \eqref{e:defcU}, we have
\begin{align}\begin{split}\label{e:am-gm}
\sum_{j=1}^{p^{d-1}}\prod_{k=0}^{p-1}\fn_k^{\frac{d-1}{d}\bmw_j(k)}
&=\sum_{(a_1,a_2,\cdots,a_{d})\in \bF_p^d,\atop a_1+a_2+\cdots+a_d=\bm0}\prod_{k=1}^d \fn_{a_k}^{\frac{d-1}{d}}\\
&\leq \sum_{(a_1,a_2,\cdots,a_{d})\in \bF_p^d,\atop a_1+a_2+\cdots+a_d=\bm0}\frac{1}{d}\sum_{l=1}^d\prod_{1\leq k\leq d, k\neq l} \fn_{a_k}=\left(\sum_{j=0}^{p-1} \fn_j\right)^{d-1}=1,
\end{split}\end{align}
where we used the AM-GM inequality. Therefore, we can conclude $d-2+\la\geq 0$, and the claim \eqref{e:neg} follows from \eqref{e:goal}. The equality in \eqref{e:neg} holds if and only if the equality in \eqref{e:am-gm} holds, which implies that
\begin{align}\label{e:equalitycond}
\sum_{l=1}^d\prod_{1\leq k\leq d, k\neq 0} \fn_{a_k}=\sum_{l=1}^d\prod_{1\leq k\leq d, k\neq 1} \fn_{a_k}=\cdots=\sum_{l=1}^d\prod_{1\leq k\leq d, k\neq p-1} \fn_{a_k},
\end{align}
for any $(a_1,a_2,\cdots, a_d)\in \bF_p^d$ with $a_1+a_2+\cdots+a_d=0$. Since by our assumption $\gcd(d,p)=1$, the only solutions for \eqref{e:equalitycond} are $\fn_0=\fn_1=\cdots=\fn_{p-1}=1/p$, or $\fn_0=1$, $\fn_1=\cdots=\fn_{p-1}=0$. This finishes the proof of Proposition \ref{p:ldpbound}. 
\end{proof}

\begin{proof}[Proof of Proposition \ref{p:ldp}]
We further decompose the set of non-equidistributed $p$-tuples $(n_0,n_1,\cdots, n_{p-1})$ into four classes:
\begin{enumerate}
\item $p$-tuples $(n_0,n_1,\cdots, n_{p-1})\in \cN$ with $\sum_{j=0}^{p-1}|n_j/n-1/p|\leq \delta$.
\item $p$-tuples $(n_0,n_1,\cdots, n_{p-1})\in \cN$  with $\fb\ln n/n<|n_0/n-1|\leq \delta$.
\item $p$-tuples $(n_0,n_1,\cdots, n_{p-1})\in \cN$  with $|n_0/n-1|\leq \fb\ln n/n$.
\item The remaining non-equidistributed $p$-tuples.
\end{enumerate}

For the first class, $\sum_{j=0}^{p-1}|n_j/n-1/p|\leq \delta$. The total number of such $p$-tuples is $e^{\OO(\ln n)}$. 
Given a $p$-tuple $(n_0,n_1,\cdots,n_{p-1})$ in the first class,
we will derive a more precise estimate of \eqref{e:neg}, by a perturbation argument. Let 
\begin{align*}\begin{split}
&\alpha_j=\frac{1}{p^{d-1}}+\varepsilon_j,\quad j=1,2,\cdots, p^{d-1}.\\
&\fn_j=\frac{1}{p}+\delta_j,\quad j=0,1,\cdots, p-1.
\end{split}\end{align*}
where $\varepsilon_1+\varepsilon_2+\cdots+\varepsilon_{p^{d-1}}=0$, and
\begin{align}\label{e:constraint2}
\varepsilon_1\bmw_1+\varepsilon_2\bmw_2+\cdots+\varepsilon_{p^{d-1}}\bmw_{p^{d-1}}=d(\delta_0,\delta_1,\cdots,\delta_{p-1}).
\end{align}
We use Taylor expansion, and rewrite \eqref{e:neg} as
\begin{align*}
-\sum_{j=1}^{p^{d-1}}\alpha_j\ln \alpha_j+(d-1)\sum_{j=0}^{p-1}\fn_j\ln \fn_j
=\frac{(d-1)p}{2}\sum_{j=0}^{p-1}\delta_j^2-\frac{p^{d-1}}{2}\sum_{j=1}^{p^{d-1}}\varepsilon_j^2+\OO\left(\sum_{j=0}^{p-1}|\delta_j|^3+\sum_{j=1}^{p^{d-1}}|\varepsilon_j|^3\right).
\end{align*}
From \eqref{e:constraint2}, we have
\begin{align*}
d^2\sum_{j=0}^{p-1}\delta_j^2
=\sum_{1\leq j,k\leq p^{d-1}}\varepsilon_i\varepsilon_j\langle\bmw_i,\bmw_j\rangle.
\end{align*}
The Gram matrix $[\langle \bmw_j,\bmw_k\rangle]_{1\leq j,k\leq p^{d-1}}$ of the vectors $\{\bmw_j\}_{1\leq j\leq p^{d-1}}$ has the same nonzero eigenvalues as the matrix
\begin{align*}
\sum_{j=1}^{p^{d-1}}\bmw_j\bmw_j^t=dp^{d-2}I_{p}+d(d-1)p^{d-3}\bm1\bm1^t,
\end{align*}
where we used \eqref{e:cov}. Thus, the Gram matrix $[\langle \bmw_j,\bmw_k\rangle]_{1\leq j,k\leq p^{d-1}}$ of the vectors $\{\bmw_j\}_{1\leq j\leq p^{d-1}}$ has an eigenvalue $d^2p^{d-2}$ corresponding to the eigenvector $(1,1,\cdots,1)$, $d-1$ eigenvalues $dp^{d-2}$, and all other eigenvalues are zero. Therefore, for $\varepsilon_1+\varepsilon_2+\cdots+\varepsilon_{p^{d-1}}=0$, we have
\begin{align*}
\sum_{j=0}^{p-1}\delta_j^2
=\frac{1}{d^2}\sum_{1\leq j,k\leq p^{d-1}}\varepsilon_i\varepsilon_j\langle\bmw_i,\bmw_j\rangle
\leq \frac{p^{d-2}}{d}\sum_{j=1}^{p^{d-1}}\varepsilon_j^2,
\end{align*}
and 
\begin{align*}
-\sum_{j=1}^{p^{d-1}}\alpha_j\ln \alpha_j+(d-1)\sum_{j=0}^{p-1}\fn_j\ln \fn_j\leq-\left(\frac{p}{2}+\oo(1)\right)\sum_{j=0}^{p-1}\delta_j^2\leq -\left(\frac{\fb p}{2}+\oo(1)\right)\frac{\ln n}{n}.
\end{align*}
The total contribution of terms in \eqref{e:noneq} satisfying $\sum_{j=0}^{p-1}|n_j/n-1/p|\leq \delta$ is bounded by
\begin{align}\label{e:case1}
\exp\left\{-\left(\frac{\fb p}{2}+\oo(1)\right)\ln n+\OO(\ln n)\right\}=\frac{\oo(1)}{n^{(d-2)}},
\end{align}
provided that we take $\fb$ sufficiently large.

For the second class, $\fb\ln n/n<|n_0/n-1|\leq \delta$. The total number of such $p$-tuples is $e^{\OO(\ln n)}$. Given a $p$-tuple $(n_0,n_1,\cdots,n_{p-1})$ in the second class, we will derive a more precise estimate of \eqref{e:neg}, by a perturbation argument. Let 
\begin{align*}\begin{split}
&\alpha_1=1-\varepsilon_1,\quad \alpha_j=\varepsilon_j,\quad j=2,3,\cdots, p^{d-1}.\\
&\fn_0=1-\varepsilon_1+\delta_0,\quad \fn_j=\delta_j,\quad j=1,2,\cdots, p-1.
\end{split}\end{align*}
where $\varepsilon_1=\varepsilon_2+\varepsilon_3+\cdots+\varepsilon_{p^{d-1}}$, and
\begin{align}\label{e:constraint3}
\varepsilon_2\bmw_2+\varepsilon_3\bmw_3+\cdots+\varepsilon_{p^{d-1}}\bmw_{p^{d-1}}=d(\delta_0,\delta_1,\cdots,\delta_{p-1}).
\end{align}
The assumption $\fb\ln n/n<|n_0/n-1|\leq \delta$ is equivalent to that $\fb\ln n/n<\delta_1+\delta_2+\cdots+\delta_{p-1}\leq \delta$.
We use Taylor expansion, and rewrite \eqref{e:neg} as
\begin{align}\begin{split}\label{e:est}
&\phantom{{}={}}-\sum_{j=1}^{p^{d-1}}\alpha_j\ln \alpha_j+(d-1)\sum_{j=0}^{p-1}\fn_j\ln \fn_j\\
&=-(d-2)\varepsilon_1+(d-1)\delta_0+\OO(\varepsilon_1^2)
-\sum_{j=2}^{p^{d-1}}\varepsilon_j\ln \varepsilon_j+(d-1)\sum_{j=1}^{p-1}\delta_j\ln \delta_j.
\end{split}\end{align}
We notice that for $2\leq j\leq p^{d-1}$, $\bmw_j(0)\leq d-2$ and $\bmw_j(1)+\bmw_j(2)+\cdots \bmw_j(p-1)\geq 2$.
From the constraint \eqref{e:constraint3}, we have
\begin{align}\label{e:dbound1}
d\delta_0=\varepsilon_2\bmw_2(0)+\varepsilon_3\bmw_3(0)+\cdots+\varepsilon_{p^{d-1}}\bmw_{p^{d-1}}(0)\leq (d-2)\varepsilon_1,
\end{align}
and 
\begin{align}\label{e:dbound2}
\delta_j\ln \left(\frac{\delta_j}{p^{d-2}}\right)
\leq\sum_{k=2}^{p^{d-1}} \frac{\bmw_k(j)}{d}\varepsilon_k\ln \varepsilon_k, \quad j=1,2,\cdots,p-1,
\end{align}
where we used the Jensen's inequality. Using \eqref{e:dbound1} and \eqref{e:dbound2}, we can upper bound \eqref{e:est} as
\begin{align*}
&\phantom{{}={}}-\sum_{j=1}^{p^{d-1}}\alpha_j\ln \alpha_j+(d-1)\sum_{j=0}^{p-1}\fn_j\ln \fn_j
\leq -\delta_0+\OO(\delta_0^2)+\frac{d-2}{2}\sum_{j=1}^{p-1}\delta_j\ln \delta_j\\
&\leq -\delta_0+\OO(\delta_0^2)+\frac{d-2}{2}\sum_{j=1}^{p-1}\delta_j\ln \left(\sum_{j=1}^{p-1}\delta_j\right)\leq -(1+\oo(1))\frac{\fb(d-2)}{2}\frac{(\ln n)^2}{n}.
\end{align*}
Thus, the total contribution of terms in \eqref{e:noneq} satisfying $\fb\ln n/n<|n_0/n-1|\leq \delta$ is bounded by
\begin{align}\label{e:case2}
\exp\left\{-(1+\oo(1))\frac{\fb(d-2)}{2}(\ln n)^2+\OO(\ln n)\right\}=\frac{\oo(1)}{n^{(d-2)}}.
\end{align}

For the third class, $|n_0/n-1|\leq \fb\ln n/n$. We denote $n_0=n-m$ and $2\leq m\leq \fb \ln n$. Fix $m$, the total number of such $p$-tuples 
\begin{align}\label{e:num}
|\{(n_1,n_2,\cdots, n_{p-1})\in \bZ_{\geq 0}^{p-1}: n_1+n_2+\cdots+n_{p-1}=m\}|=\OO(p^m).
\end{align}
Given a $p$-tuple $(n_0,n_1,\cdots,n_{p-1})$, in the third class with $n_0=n-m$ and $m\leq \fb\ln n$, we reestimate the first factor on the righthand side of \eqref{e:expldp},
\begin{align}\begin{split}\label{e:entr}
{n\choose n_0,n_1,\cdots, n_{p-1}}{dn\choose dn_0,dn_1,\cdots, dn_{p-1}}^{-1}
\leq \OO(1)d^m\left(\frac{dm}{en}\right)^{(d-1)m}.
\end{split}\end{align}
For the number of walk paths in \eqref{e:expldp}, we notice that $\bmw_j(1)+\bmw_j(2)+\cdots \bmw_j(p-1)\geq 2$ for $2\leq j\leq p^{d-1}$. Moreover, if $\bmu_1+\bmu_2+\cdots+\bmu_n=(dn_0, dn_{1},\cdots, dn_{p-1})$, then $dn_1+dn_2+\cdots+dn_{p-1}=dm$.
Therefore,  $\bmu_i=\bmw_1$ for all $1\leq i\leq n$, except for at most $dm/2$ of them, and we have
\begin{align}\begin{split}\label{e:pathest}
|\{(\bmu_1,\bmu_2\cdots, \bmu_n)\in \cU_{d,p}^n: \bmu_1+\bmu_2+\cdots+\bmu_n=(dn_0, dn_{1},\cdots, dn_{p-1})\}|
\leq \left(p^{d-1}n\right)^{dm/2}.
\end{split}\end{align}
Putting \eqref{e:num}, \eqref{e:entr} and \eqref{e:pathest} together, the total contribution of terms in \eqref{e:noneq} satisfying $|n_0/n-1|\leq \fb\ln n/n$ is bounded by
\begin{align}\label{e:case3}
\sum_{m=2}^{\fb \ln n}\OO(1)p^m d^m\left(\frac{dm}{en}\right)^{(d-1)m}\left(p^{d-1}n\right)^{dm/2}=\frac{\OO(1)}{n^{(d-2)}}.
\end{align}

For the last class, the total number of such $p$-tuples is $e^{\OO(\ln n)}$, and each term is exponentially small, i.e. $e^{-c(\delta)n}$. Therefore the total contribution is
\begin{align}\label{e:case4}
\exp\{-c(\delta) n+\OO(\ln n)\}=\frac{\oo(1)}{n^{(d-2)}}.
\end{align}

The claim \eqref{e:noneq} follows from combining the discussion of all four cases, \eqref{e:case1}, \eqref{e:case2}, \eqref{e:case3} and \eqref{e:case4}.
\end{proof}
\bibliography{References.bib}{}
\bibliographystyle{abbrv}

\end{document}